\documentclass[11pt, oneside]{article}   	
\usepackage{amsmath, amsthm, amsfonts, amssymb}
\usepackage{amsmath, amsfonts, amssymb}

\newcommand{\Normal}{{\mathcal N}}
\newcommand{\N}{{\mathbb N}}

\newcommand{\R}{{\mathbb R}}
\newcommand{\C}{{\mathbb C}}

\newcommand{\E}{{\mathbb E}}
\newcommand{\Prob}{{\mathbb P}}

\newcommand{\dto}{\overset{d}{\to }}

\DeclareMathOperator{\Image}{Im}

\DeclareMathOperator{\Var}{Var}
\DeclareMathOperator{\Hess}{Hess}

\newtheorem{theorem}{Theorem}[section]
\newtheorem{corollary}[theorem]{Corollary}
\newtheorem{lemma}[theorem]{Lemma}

\theoremstyle{definition}

\theoremstyle{remark}
\newtheorem{remark}[theorem]{Remark}

\long\def\symbolfootnote[#1]#2{\begingroup%
\def\thefootnote{\fnsymbol{footnote}}\footnote[#1]{#2}\endgroup}

\begin{document}

\title{On spectral measures of random Jacobi matrices\thanks{This work is supported by JSPS Grant-in-Aid for Young Scientists (B) no.~16K17616}
}

\author{Trinh Khanh Duy}

\maketitle

\begin{abstract}
The paper studies the limiting behavior of spectral measures of random Jacobi matrices of Gaussian, Wishart and MANOVA beta ensembles. We show that the spectral measures converge weakly to a limit distribution  which is the semicircle distribution, Marchenko-Pastur distributions or Kesten-Mckey distributions, respectively. The Gaussian fluctuation around the limit is then investigated.


\symbolfootnote[0]{Mathematical Subject Classification (2000). 60B20; 60F05}
\end{abstract}

\section{Introduction}
Three classical random matrix ensembles on the real line, Gaussian beta ensembles, Wishart beta ensembles and MANOVA beta ensembles, are now realized as eigenvalues of certain random Jacobi matrices. For instance, the following random Jacobi matrices whose components are independent and distributed as 
\[
	H_{n,\beta} = \begin{pmatrix}
		a_1		&b_1		\\
		b_1		&a_2		&b_2\\
		&\ddots	&\ddots	&\ddots\\
		&		&b_{n - 1}	&a_n
	\end{pmatrix}
	\sim\frac{1}{\sqrt{n\beta}}\begin{pmatrix}
		\Normal(0,2)		&\chi_{(n-1)\beta}		\\
		\chi_{(n - 1)\beta}		&\Normal(0,2)		&\chi_{(n - 2)\beta}\\
		&\ddots	&\ddots	&\ddots\\
		&		&\chi_{\beta}	&\Normal(0,2)
	\end{pmatrix}
\]
are matrix models of (scaled) Gaussian beta ensembles for any $\beta > 0$. Here $\Normal(\mu, \sigma^2)$ denotes the normal (or Gaussian) distribution with mean $\mu$ and variance $\sigma^2$, and $\chi_k$ denotes the chi distribution with $k$ degrees of freedom. Namely, the eigenvalues of $H_{n,\beta}$ are distributed as Gaussian beta ensembles,
\[
	(\lambda_1, \dots, \lambda_n) \propto |\Delta(\lambda)|^{\beta} \exp\left(-\frac{n\beta}{4} \sum_{i = 1}^n \lambda_i^2 \right),
\]
where $\Delta(\lambda) = \prod_{i<j}(\lambda_j - \lambda_i)$ denotes the Vandermonde determinant.

Three special values of beta, $\beta = 1,2$ and $4$, correspond to Gaussian orthogonal, unitary and symplectic ensembles (GOE, GUE and GSE) in which the above formula describes the joint distribution of eigenvalues of random matrices with real, complex and quaternion entries, respectively. As a generalization, Gaussian beta ensembles were originally defined as ensembles of points on the real line whose joint density function is given as above. They can be also viewed as the equilibrium measure of a one dimensional Coulomb log-gas at the inverse temperature $\beta$. Using the idea of tridiagonalizing a GOE matrix, Dumitriu and Edelman  \cite{Dumitriu-Edelman-2002} introduced the  model $H_{n, \beta}$ for Gaussian beta ensembles. In the same paper, they also gave a matrix model for Wishart beta ensembles. A model for MANOVA beta ensembles was discovered later by  Killip and Nenciu \cite{Killip-Nenciu-2004}.

One of main objects in random matrix theory is to study the limiting behavior of the empirical distribution of eigenvalues 
\[
	L_n = \frac{1}{n} \sum_{i = 1}^n \delta_{\lambda_i},
\]
where $\delta$ denotes the Dirac measure. For Gaussian beta ensembles, as $n$ tends to infinity, the empirical distributions  converge weakly, almost surely, to the semicircle distribution, which is well known as Wigner's semicircle law. The convergence means that for any bounded continuous function $f$ on $\R$, 
\[
	\langle L_n, f\rangle = \frac{1}{n} \sum_{i = 1}^n f(\lambda_i) \to \langle sc, f\rangle \text{ almost surely as } n\to \infty,
\]
with $sc$ denoting the semicircle distribution, a probability measure supported on $[-2,2]$ with density $sc(x) = (2\pi)^{-1} \sqrt{4 - x^2}$. A fluctuation around the limit was also investigated. To be more precise, it was shown that for a `nice' test function $f$,   
\[
	n(\langle L_n, f\rangle - \langle sc, f\rangle ) = \sum_{i = 1}^n (f(\lambda_i) - \langle sc, f\rangle) \dto \Normal(0, a_f^2),
\] 
where $a_f^2$ can be written as a quadratic functional of $f$. There are several ways to prove those results. See Johansson~\cite{Johansson-1998} for an approach based on joint density function, Dumitriu and Edelman \cite{Dumitriu-Edelman-2006} and  Dumitriu and Paquette \cite{Dumitriu-Paquette-2012} for a combinatorial approach based on the random Jacobi matrix models. 
Since GOE and GUE have their original matrix models, we can see more approaches in books \cite{Anderson-book,Pastur-book}. Note that the idea in the last section of this paper is also applicable to study such Gaussian fluctuation. Using the idea, we can show that the class of `nice' test functions for which the above central limit theorem holds contains at least differentiable functions whose derivative is continuous of polynomial growth. 


The spectral measures of random Jacobi matrices associated with those beta ensembles have been investigated recently. The weak convergence to a limit distribution, a central limit theorem for moments and large deviations have been established \cite{Dette-Nagel-2012,Gamboa-Rouault-2011,Nagel-2013}. The spectral measure of a finite Jacobi matrix, a symmetric tridiagonal matrix of the form,
\[
	J = \begin{pmatrix}
		a_1		&b_1		\\
		b_1		&a_2		&b_2\\
		&\ddots	&\ddots	&\ddots\\
		&		&b_{n - 1}	&a_n
	\end{pmatrix},
	(a_i \in \R, b_i > 0),
\]
is defined to be a unique probability measure $\mu$ on $\R$ satisfying 
\[
	\langle \mu, x^k\rangle = \langle J^k e_1, e_1\rangle = J^k(1,1), k = 0,1,\dots,
\]
where $e_1 = (1,0,\dots,0)^t \in \R^n$. Let $\{\lambda_1, \dots, \lambda_n\}$ be the eigenvalues of $J$ and $\{v_1, \dots, v_n\}$ be the corresponding eigenvectors which are chosen to be an orthonormal basis of $\R^n$. Then the spectral measure $\mu$ can be written as 
\[
	\mu = \sum_{i = 1}^n q_i^2 \delta_{\lambda_i},\quad q_i = |v_i(1)|.
\]
Note that the eigenvalues $\{\lambda_i\}$ are distinct and the weights $\{q_i^2\}$ are all positive. Moreover, a finite Jacobi matrix of size $n$ is one-to-one correspondence with a probability measure supported on $n$ real points.

Let $\mu_n$ be the spectral measure of $H_{n, \beta}$,
\[
	\mu_n = \sum_{i = 1}^n q_i^2 \delta_{\lambda_i}.
\]
In this case, and in all three beta ensembles in the paper, the weights $\{q_i^2\}$ are independent of eigenvalues and have Dirichlet distribution with parameters $(\beta/2, \dots, \beta/2)$ (or symmetric Dirichlet distribution with parameter $\beta/2$). The distribution of $\{q_i^2\}$ is the same as that of the vector 
\[
	\left(\frac{\chi_{\beta,1}^2}{\sum_{i = 1}^n \chi_{\beta,i}^2} , \dots, \frac{\chi_{\beta,n}^2}{\sum_{i = 1}^n \chi_{\beta,i}^2} \right),
\]
where $\{\chi_{\beta,i}^2\}_{i = 1}^n$ is an i.i.d.~sequence of random variables having chi-squared distributions with $\beta$ degrees of freedom. In connection with empirical distributions, Nagel~\cite{Nagel-082013} showed that as $n$ tends to infinity, the Kolmogorov distance between $L_n$ and $\mu_n$ converges almost surely to zero. Thus the spectral measures and the empirical distributions converge to the same limit. Moreover, the limiting behavior of spectral measures of Jacobi matrices can be read off from the convergence of their entries. For Gaussian beta ensembles, it is clear that  
\[
	H_{n, \beta} 
	\to 
	\begin{pmatrix}
		0	&	1\\
		1	&	0	&	1\\
		&\ddots	&\ddots	&\ddots
	\end{pmatrix} =: J_{free}, \text{ almost surely as } n \to \infty.
\] 
Here the convergence means  the piecewise convergence of entries. The non-random Jacobi matrix $J_{free}$ is called the free Jacobi matrix whose spectral measure is nothing but the semicircle distribution \cite[Section~1.10]{Simon-book}. Consequently the spectral measures $\mu_n$, and hence, the empirical distributions $L_n$, converge weakly, almost surely, to the semicircle distribution. This result may be regarded as a strong law of large numbers for spectral measures.


A natural problem now is to study the fluctuation of spectral measures around the limit, or a central limit theorem for $\langle \mu_n, f \rangle$ with a `nice' function $f$. In a work which is not so related to random matrix theory, Dette and Nagel \cite{Dette-Nagel-2012} derived a central limit theorem for moments $\{\langle \mu_n, x^k \rangle\}$ of spectral measures. The result covers Gaussian,  Wishart beta ensembles and MANOVA beta ensembles with fixed parameters. The aim of this paper is to reconsider the central limit theorem. We propose a universal approach which can be easily applied to all three beta ensembles. The idea is that for a polynomial test function, when $n$ is large enough, $
	\langle \mu_n, p\rangle = p(H_{n,\beta}) (1,1)   
$
is a polynomial of finite variables. Then the central limit theorem follows from the limiting behavior of entries of Jacobi matrices. Furthermore, by a relation between spectral measures and empirical measures, we obtain an explicit formula for the limit variance and can extend the central limit theorem to a large class of test functions. Our main result for Gaussian beta ensembles can be stated as follows. 
 \begin{theorem}
\begin{itemize}
	\item[\rm (i)] The spectral measures $\mu_n$ converge weakly, almost surely, to the semicircle distribution as $n \to \infty$, that is, for any bounded continuous function $f$, 
	\[
		\langle \mu_n, f\rangle = \sum_{i = 1}^n q_i^2 f(\lambda_i) \to \langle sc, f\rangle \text{ almost surely as }n \to \infty.
	\]
	\item[\rm (ii)] For a function $f$ with continuous derivative of polynomial growth,
	\[
	\frac{\sqrt{n\beta}}{\sqrt{2}} (\langle \mu_n, f \rangle - \E[\langle \mu_n, f \rangle]) \dto \Normal(0, \sigma^2(f)) \text{ as } n \to \infty,
	\]
where $\sigma^2(f) = \langle sc, f^2 \rangle - \langle sc, f \rangle^2 = \Var_{sc}[f]$. Here `$\dto$' denotes convergence in distribution or weak convergence of random variables.
\end{itemize}
\end{theorem}
\noindent The results for Wishart and MANOVA beta ensembles are analogous where the semicircle distribution is replaced by Marchenko-Pastur distributions and Kesten-Mckey distributions, respectively.

The Kolmogorov distance, the metric which implies the weak convergence, between two measures $\mu$ and $\nu$ on the real line with distribution functions $F_\mu$ and $F_\nu$, respectively, is defined by 
\[
	d_K(\mu, \nu) = \sup_{x \in \R} |F_\mu(x) - F_\nu(x)|. 
\]
Here $F_\mu(x) = \mu((-\infty, x])$. As mentioned above, for fixed $\beta > 0$, the Kolmogorov distance between the empirical distributions $L_n$ and the spectral measures $\mu_n$ converges to zero almost surely as $n$ tends to infinity. For the proof, we only need properties that both measures are support on the set of eigenvalues and that the weights $\{q_i^2\}$ have symmetric Dirichlet distribution with parameter $\beta/2$ \cite[Theorem~4.2]{Nagel-082013}. Consequently, the following strong law of large numbers for empirical distributions holds.

\begin{corollary}
	For all three beta ensembles in this paper, as $n \to \infty$, the empirical distributions $L_n$ converge weakly, almost surely, to the same limit as the spectral measures. 
\end{corollary} 
 
The paper is organized as follows. In the next section,  we consider general random Jacobi matrices and derive the weak convergence of spectral measures as well as the central limit theorem for polynomial test functions. Applications to Gaussian, Wishart and MANOVA beta ensembles are then investigated in turn. The last section is devoted to extend the central limit theorem to a larger class of test functions.

\section{Limiting behavior of spectral measures of random Jacobi matrices}
Let us begin by introducing some spectral properties of (non random) Jacobi matrices. A semi-infinite Jacobi matrix is a symmetric tridiagonal matrix of the form
\[
	J = \begin{pmatrix}
		a_1		&b_1		\\
		b_1		&a_2			&b_2\\
		&\ddots	&\ddots		&\ddots
	\end{pmatrix},
	\text{ where }a_i \in \R, b_i > 0.
\] 
To a Jacobi matrix $J$, there exists a probability measure $\mu$ such that 
\[
	\langle \mu, x^k\rangle = \int_\R x^k d\mu = \langle J^k e_1, e_1\rangle, k = 0,1,\dots,
\]
where $e_1 = (1,0,\dots,)^t\in \ell^2$. Then $\mu$ is unique, or $\mu$ is determined by its moments, if and only if, $J$ is an essentially self-adjoint operator on $\ell^2$. If the parameters $\{a_i\}$ and $\{b_i\}$ are bounded, or more generally, if $\sum b_i^{-1} = \infty$, then $J$ is essentially self-adjoint \cite[Corollary 3.8.9]{Simon-book}. In case of uniqueness, we call $\mu$ the spectral measure of $J$, or of $(J, e_1)$.  See \cite[Chapter~2]{Deift-book} or \cite[Section~3.8]{Simon-book} for more details on Jacobi matrices.

When $J$ is a finite Jacobi matrix of order $n$, then the spectral measure $\mu$ is supported on the eigenvalues $\{\lambda_i\}$ of $J$ with weights $\{q_i^2\} = \{v_i(1)^2\}$,
\[
	\mu = \sum_{i = 1}^n q_i^2 \delta_{\lambda_i}.
\] 
Here $\{v_1, \dots, v_n\}$ are the corresponding eigenvectors which are chosen to be an orthogonal basis of $\R^n$. The eigenvalues $\{\lambda_i\}_{i = 1}^n$ are distinct and the weights $\{q_i^2\}_{i = 1}^n$ are all positive \cite[Proposition~2.40]{Deift-book}.

We are now in a position to study the convergence of spectral measures of Jacobi matrices. Recall that spectral measures are defined by their moments. Does the convergence of moments imply the weak convergence of probability measures? The following lemma gives us the answer. It is a classical result which can be found in some textbooks in probability theory. 
\begin{lemma}\label{lem:deterministic}
Assume that $\{\mu_n\}_{n=1}^\infty$ and $\mu$ are probability measures on $\R$ such that for all $k=0,1,\dots,$
\[
	\langle \mu_n, x^k\rangle \to \langle \mu, x^k\rangle	\text{ as } n\to \infty.
\]
Assume further that the measure $\mu$ is determined by its moments. Then $\mu_n$ converges weakly to $\mu$ as $n \to \infty$. Moreover, if $f$ is a continuous function of polynomial growth, that is, there is a polynomial $p$ such that $|f(x)|\le p(x)$ for all $x\in \R$, then we also have 
\[
	\langle \mu_n, f\rangle \to \langle \mu, f\rangle	\text{ as } n\to \infty.
\]
\end{lemma}
\begin{proof}
	The first part of this lemma is a well-known moment problem \cite[Theorem~30.2]{Billingsley-PnM}. The second part follows by a truncated argument. For the sake of completeness, we give proof here. Assume that the sequence $\{\mu_n\}$ converges weakly to $\mu$ and that $\langle \mu_n, p \rangle$ converges to $\langle \mu, p\rangle$ for all polynomials $p$. Let $f$ be a continuous function which is dominated by a polynomial $p$, $|f(x)| \le p(x)$ for all $x \in \R$. For $M > 0$, write $f_M$ for the truncated function 
\[
	f_M(x) = \begin{cases}
		-M, &\text{if } f(x) \le -M,\\
		f(x), &\text{if } |f(x)| \le M,\\
		M, &\text{if } f(x) \ge M.\\
	\end{cases}
\]
Then it is clear that $|f - f_M| \le p - p_M$, where $p_M$ is the truncated function of $p$. Thus by the triangle inequality, 
\begin{align*}
	\left| \langle \mu_n, f \rangle - \langle \mu, f \rangle\right| &\le \left| \langle \mu_n, f - f_M \rangle \right| + \left| \langle \mu_n, f_M \rangle - \langle \mu, f_M \rangle\right| + \left| \langle \mu, f - f_M \rangle\right|  \\
	&\le   \langle \mu_n, p - p_M \rangle + \left| \langle \mu_n, f_M \rangle - \langle \mu, f_M \rangle\right| + \langle \mu, p - p_M \rangle  \\
	&=  \langle \mu_n, p \rangle - \langle \mu_n, p_M \rangle + \left| \langle \mu_n, f_M \rangle - \langle \mu, f_M \rangle\right| + \langle \mu, p - p_M \rangle.
\end{align*}
As $n \to \infty$, the first term converges to $\langle \mu, p\rangle$ by the assumption, the second term converges to $\langle \mu, p_M\rangle$ and the third term converges to $0$ because $p_M$ and $f_M$ are bounded continuous functions. Therefore 
\[
	\limsup_{n \to \infty} \left| \langle \mu_n, f \rangle - \langle \mu, f \rangle\right| \le 2 \langle \mu, p - p_M \rangle.
\]
Finally, by letting $M \to \infty$, $ \langle \mu, p - p_M \rangle \to 0$ by the monotone convergence theorem. The lemma is proved.
\end{proof}

To random probability measures, we deal with two types of convergence, almost sure convergence and convergence in probability. A result for almost sure convergence is a direct consequence of the above deterministic result. However, it is not the case for convergence in probability. When the limit measure has compact support, the following result on convergence in probability of random measures can be derived by a method of polynomials approximation, see subsection~2.1.2 in \cite{Anderson-book}, for instance.  
\begin{lemma}\label{lem:convergence-of-moments}
	Let $\{\mu_n\}_{n = 1}^\infty$ be a sequence of random probability measures and $\mu$ be a non-random probability measure which is determined by its moments. Assume that any moment of $\mu_n$ converges almost surely to that of $\mu$, that is, 
	 for any $k = 0,1,\dots,$
\[
	\langle \mu_n, x^k\rangle \to \langle \mu, x^k\rangle	\text{ a.s.~as } n\to \infty.
\]  
Then as $n\to \infty$, the sequence of measures $\{\mu_n\}$ converges weakly, almost surely, to $\mu$, namely, for any bounded continuous function $f$, 
\[
	\langle \mu_n, f\rangle \to \langle \mu, f\rangle	\text{ a.s.~as } n\to \infty.
\]
The convergence still holds for a continuous function  $f$ of polynomial growth. An analogous result holds for convergence in probability.
\end{lemma}
\begin{proof}
	The case of almost sure convergence is a direct consequence of Lemma~\ref{lem:deterministic}. Indeed, for $k \ge 1$, let 
\[
	A_k = \{\omega : \langle \mu_n(\omega) , x^k\rangle \to \langle \mu, x^k\rangle \text{ as } n \to \infty\}.
\]	
Then $\Prob(A_k) = 1$ by the assumption. Therefore $\Prob(A:= \bigcap_{k=1}^\infty A_k) = 1$. Applying Lemma~\ref{lem:deterministic} to the sequence of probability measures $\{\mu_n(\omega)\}$, for $\omega \in A$, yields the desired result.

Next we consider the case of convergence in probability. The idea here is to use the following criterion for convergence in probability \cite[Theorem~20.5]{Billingsley-PnM}: a sequence $\{X_n\}$ converges to $X$ in probability if and only if for every subsequence $\{X_{n(m)}\}$, there is a further subsequence $\{X_{n(m_k)}\}$ that converges almost surely to $X$. Let $f$ be a continuous function which is dominated by some polynomial. Given a subsequence $\{n(m)\}$, the aim now is to find a subsequence $n(m_k)$ such that $\{\langle \mu_{n(m_k)}, f\rangle \}$ converges almost surely to $\langle \mu, f\rangle$. Let $\{n(0,m) = n(m)\}$. For $k \ge 1$, using the necessary condition in the criterion, we can find a subsequence $\{n(k,m)\}$ of $\{n(k-1,m)\}$  such that 
\[
	\langle \mu_{n(k, m)}, x^k\rangle \to \langle \mu, x^k\rangle	\text{ a.s.~as } n\to \infty.
\]
By selecting the diagonal, we get a subsequence $\{n(m_k) = n(k, k)\}$ for which all moments of $\{\mu_{n(m_k)}\}$ converge almost surely to the corresponding moments of $\mu$. Consequently, the sequence $\{\mu_{n(m_k)}\}$ converges weakly, almost surely, to $\mu$ by the first part of this lemma, which implies that $\{\langle \mu_{n(m_k)}, f\rangle \} \to \langle \mu, f\rangle$ almost surely. The proof is complete.
\end{proof}

Let us now explain the main idea of this paper. Consider the sequence of random Jacobi matrices
\[
	J_n = \begin{pmatrix}
		a_1^{(n)}		&b_1^{(n)}		\\
		b_1^{(n)}		&a_2^{(n)}		&b_2^{(n)}		\\
		&\ddots	&\ddots	&\ddots	\\
		&&		b_{n - 1}^{(n)}	&a_n^{(n)}
	\end{pmatrix},
\]
and let $\mu_n$ be the spectral measure of $(J_n, e_1)$. Assume that each entry of $J_n$ converges almost surely to a non random limit as $n \to \infty$, that is, for any fixed $i$, as $n \to \infty$,
\begin{equation}\label{almost-sure-assumption}
	a_i^{(n)} \to \bar a_i; \quad
	b_i^{(n)} \to \bar b_i \text{ a.s.}
\end{equation}
Here we require that $\bar a_i$ and $\bar b_i$ are non random and $\bar b_i >0$.  Assume further that the spectral measure of $(J_\infty, e_1)$, denoted by $\mu_\infty$, is unique, where $J_\infty$ is the infinite Jacobi matrix consisting of $\{\bar a_i\}$ and $\{\bar b_i\}$,
\[
	J_\infty = \begin{pmatrix}
		\bar a_1		&\bar b_1		\\
		\bar b_1		&\bar a_2		&\bar b_2	\\
		&\ddots		&\ddots		&\ddots
	\end{pmatrix}.
\]
Then the measure $\mu_\infty$ is determined by its moments, and hence we get the following result.
\begin{theorem}
	The spectral measures $\mu_n$ converge weakly, almost surely, to the limit measure $\mu_\infty$ as $n \to \infty$. If in the assumption \eqref{almost-sure-assumption}, convergence in probability is assumed instead of almost sure convergence, then the spectral measures $\mu_n$ converge weakly, in probability, to $\mu_\infty$ as $n \to \infty$.
\end{theorem}

\begin{remark}
	These results may be referred to as weak and strong laws of large numbers for spectral measures of random Jacobi matrices. They are natural results which may be found somewhere. For instance, the strong law was mentioned in \cite{Nagel-082013}.
\end{remark}
\begin{proof}
	Let $p$ be a polynomial of degree $m$. When $n$ is large enough,  
$
	\langle \mu_n, p\rangle = p(J_n) (1,1)
$
is a polynomial of $\{a_i^{(n)}, b_i^{(n)}\}_{i = 1, \dots, \lceil \frac m2 \rceil}$. Therefore, as $n \to \infty$, 
\[
	\langle \mu_n, p\rangle \to \langle \mu_\infty, p\rangle \text{ almost surely (resp.~in probability)},
\]
which implies the weak convergence of $\mu_n$ by Lemma~\ref{lem:convergence-of-moments}.
\end{proof}

Next, we consider the second order of the convergence of spectral measures, or a type of central limit theorem. It turns out that the central limit theorem for polynomial test functions is  a direct consequence of a joint central limit theorem for entries of Jacobi matrices. Indeed, assume that there are random variables $\{\eta_i\}$ and $\{\zeta_i\}$ defined on the same probability space such that for some fixed $r > 0$, for any $i$, as $n \to \infty$,
\begin{equation}\label{weak-convergence-assumption}
\begin{aligned}
	\tilde a_i^{(n)} =  n^{r}(a_i^{(n)} - \bar a_i) \dto \eta_i, \\
	\tilde b_i^{(n)} =n^{r}(b_i^{(n)} - \bar b_i) \dto \zeta_i.
\end{aligned}
\end{equation}
Moreover, we assume that the joint weak convergence holds. This means that any finite linear combination of $\tilde a_i^{(n)}$ and $\tilde b_i^{(n)}$ converges weakly to the corresponding linear combination of $\eta_i$ and $\zeta_i$ as $n \to \infty$,  namely, for any real numbers $c_i$ and $d_i$,
\[
	\sum_{finite} (c_i \tilde a_i^{(n)} + d_i \tilde b_i^{(n)}) \dto \sum_{finite} (c_i \eta_i + d_i \zeta_i).
\]
From now on, both conditions~\eqref{almost-sure-assumption} and \eqref{weak-convergence-assumption} will be  written in a compact form
\[
	J_n  \approx \begin{pmatrix}
		\bar a_1		&\bar b_1		\\
		\bar b_1		&\bar a_2		&\bar b_2	\\
		&\ddots		&\ddots		&\ddots
	\end{pmatrix}
	+\frac{1}{n^r} \begin{pmatrix}
		\eta_1		&\zeta_1		\\
		\zeta_1		&\eta_2		&\zeta_2	\\
		&\ddots		&\ddots		&\ddots
	\end{pmatrix},
\]
or in term of entries 
\begin{align*}
	a_i^{(n)} \approx \bar a_i + \frac{1}{n^r} \eta_i,\\
	b_i^{(n)} \approx \bar b_i + \frac{1}{n^r} \zeta_i.\\
\end{align*}

Let $f$ be a polynomial of $2k$ variables $(a_1,\dots, a_k, b_1, \dots, b_k)$. For simplicity, we write $f(a_i, b_i)$ instead of $f(a_1, \dots, a_k, b_1, \dots, b_k)$.
\begin{lemma}\label{lem:polynomial}
	\begin{itemize}
	\item [\rm (i)] As $n \to \infty$,
	\begin{align*}
		n^r \left( f(a_i^{(n)}, b_i^{(n)})  - f(\bar a_i, \bar b_i)\right) - \sum_{i = 1}^k \left(\frac{\partial f}{\partial a_i}(\bar a_i, \bar b_i)  \tilde a_i^{(n)} + \frac{\partial f}{\partial b_i}(\bar a_i, \bar b_i) \tilde b_i^{(n)} \right) &\to 0 \\
		&\text{in probability}.
	\end{align*}
	\item[\rm (ii)] As $n \to \infty$,
	\[
		n^r \left( f(a_i^{(n)}, b_i^{(n)})  - f(\bar a_i, \bar b_i)\right)  \dto \sum_{i = 1}^k \left(\frac{\partial f}{\partial a_i}(\bar a_i, \bar b_i)  \eta_i + \frac{\partial f}{\partial b_i}(\bar a_i, \bar b_i) \zeta_i \right).
	\]
	\end{itemize}
\end{lemma}
\begin{proof}
	Write 
	\[
		a_i^{(n)} = \bar a_i + \frac{1}{n^r} \tilde a_i^{(n)}; b_i^{(n)} = \bar b_i + \frac{1}{n^r} \tilde b_i^{(n)}. 
	\]
	Then use the Taylor expansion of $f(a_i^{(n)}, b_i^{(n)})$ at $(\bar a_i, \bar b_i)$ with noting that the Taylor expansion of a polynomial consists of finitely many terms,
	\begin{align*}
		f(a_i^{(n)}, b_i^{(n)}) = f(\bar a_i, \bar b_i) + \frac{1}{n^r} \sum_{i = 1}^k \left(\frac{\partial f}{\partial a_i}(\bar a_i, \bar b_i)  \tilde a_i^{(n)} + \frac{\partial f}{\partial b_i}(\bar a_i, \bar b_i) \tilde b_i^{(n)} \right)  + \sum{}^*.
	\end{align*}
Each term in the finite sum $\sum{}^*$ has the following form, 
	\[
		c(\alpha, \beta) \prod_{i = 1}^k (a_i^{(n)} - \bar a_i)^{\alpha_i}(b_i^{(n)} - \bar b_i)^{\beta_i}, 
	\]
	where $\{\alpha_i\}$ and $\{\beta_i\}$ are non negative integers and $\sum_{i = 1}^k(\alpha_i + \beta_i) \ge 2$. Therefore, when that term is multiplied by $n^r$, it converges to $0$ in distribution, and hence, in probability by Slutsky's theorem.

By using Slutsky's theorem again,	we see that (ii) is a consequence of (i). The proof is complete.
\end{proof}

Let $p$ be a polynomial of degree $m>0$. Then there is a polynomial of $2\lceil \frac m2 \rceil$ variables such that for $n > m/2$, 
\[
	\langle \mu_n, p\rangle = p(J_n) (1,1) = f(a_1^{(n)},\dots, a_{\lceil \frac m2 \rceil}^{(n)}, b_1^{(n)}, \dots, b_{\lceil \frac m2 \rceil}^{(n)}).
\]
Therefore, by Lemma~\ref{lem:polynomial}, we obtain the central limit theorem for polynomial test functions.
\begin{theorem}\label{thm:weak-convergence}
	For any polynomial $p$, 
$
		n^{r} \left( \langle  \mu_n, p \rangle- \langle \mu_\infty, p\rangle\right)
$
	converges weakly to a limit $\xi_\infty(p)$ as $n \to \infty$.
\end{theorem}

	Since we do not assume that all moments of $\{a_i^{(n)}\}$ and $\{b_i^{(n)}\}$ are finite, even the expectation of $\langle  \mu_n, p \rangle$, $\E[\langle  \mu_n, p \rangle]$ may not exist. Thus we need further assumptions to ensure the convergence of mean and variance in the central limit theorem above. Our assumptions are based on the following basic result in probability theory (the corollary following Theorem~25.12 in  \cite{Billingsley-PnM}).
\begin{lemma}\label{lem:convergence-of-variance}
Assume that the sequence $\{X_n\}$ converges weakly to a random variable $X$. If for some $\delta > 0$,
\[
	\sup_{n} \E[|X_n|^{2 + \delta}] < \infty,
\] 
then $\E[X_n] \to \E[X]$ and $\Var[X_n] \to \Var[X]$ as $n \to \infty$. In general, if $X_n$ converges to $X$ in probability and for $q \ge 1$ and $\delta > 0$,
\[
	\sup_{n} \E[|X_n|^{q + \delta}] < \infty, 
\]
then $X_n$ converges to $X$ in $L_q$. 
\end{lemma}

We make the following assumptions 
\begin{itemize}
	\item[(i)] all moments of $\{a_i^{(n)}\}$ and $\{b_i^{(n)}\}$ are finite and in addition, the convergences in \eqref{almost-sure-assumption} also hold in $L_q$ for all $q <\infty$, which is equivalent to the following conditions
	\begin{equation}\label{Lq-bounded}
		\sup_n \E[|a_i^{(n)}|^k] < \infty, \quad \sup_n \E[|b_i^{(n)}|^k] < \infty, \text{ for all $k=1,2,\dots$};
	\end{equation}
	\item[(ii)] $\E[\eta_i] = 0, \E[\zeta_i] = 0$, and for some $\delta > 0$, 
	\begin{equation}\label{variance-bounded}
		\sup_{n} \E[|\tilde a_i^{(n)}|^{2 + \delta}] < \infty, \quad \sup_{n} \E[|\tilde b_i^{(n)}|^{2 + \delta}] < \infty.
	\end{equation}
\end{itemize}
\begin{lemma}\label{lem:mean-convergence}
 As $n \to \infty$,
\begin{align} 
	n^{2r}\E\left [\left (f(a_i^{(n)}, b_i^{(n)}) - f(\bar a_i, \bar b_i) \right)^2 \right] &\to \Var \left[ \sum_{i = 1}^k \left(\frac{\partial f}{\partial a_i}(\bar a_i, \bar b_i)  \eta_i + \frac{\partial f}{\partial b_i}(\bar a_i, \bar b_i) \zeta_i \right)\right], \\
	n^r \left(\E[f(a_i^{(n)}, b_i^{(n)})] - f(\bar a_i, \bar b_i) \right) &\to 0,\\
	n^{2r}\Var\left [f(a_i^{(n)}, b_i^{(n)}) \right] &\to \Var \left[ \sum_{i = 1}^k \left(\frac{\partial f}{\partial a_i}(\bar a_i, \bar b_i)  \eta_i + \frac{\partial f}{\partial b_i}(\bar a_i, \bar b_i) \zeta_i \right)\right].
\end{align}

\end{lemma}
\begin{proof}
It is just a direct consequence of Lemma~\ref{lem:convergence-of-variance}.
\end{proof}

We state now a slightly different form of the central limit theorem for polynomial test functions.
\begin{theorem}\label{thm:general}
Under assumptions~\eqref{almost-sure-assumption}--\eqref{variance-bounded},
for any polynomial $p$, as $n \to \infty$,
\[
	\langle  \mu_n, p \rangle \to \langle  \mu_\infty, p \rangle \text{ almost surely and in $L^q$ for all $q <\infty$};
\]
	\[
		n^{r} ( \langle  \mu_n, p \rangle- \E[\langle \mu_n, p\rangle]) \dto \xi_\infty(p).
	\]
Moreover, $\E[\xi_\infty(p)] = 0$ and  
\[
	n^{2r}\Var[\langle  \mu_n, p \rangle] \to \Var[\xi_\infty(p)] \text{ as $n \to \infty$}.
\]

\end{theorem}
%
%
%
%
%

\section{Gaussian beta ensembles or $\beta$-Hermite ensembles}
Let $H_{n,\beta}$ be a random Jacobi matrix whose elements are independent (up to the symmetric constraint) and are distributed as
\[
	H_{n,\beta} = \frac{1}{\sqrt{n\beta}}\begin{pmatrix}
		\Normal(0,2)		&\chi_{(n-1)\beta}	\\
		\chi_{(n-1)\beta}	&\Normal(0,2)		&\chi_{(n-2)\beta}		\\
					
												&\ddots		&\ddots		&\ddots \\
						
						&&					\chi_{\beta}	&\Normal(0,2)
	\end{pmatrix}.
\]
Then the eigenvalues $\{\lambda_i\}$ of $H_{n,\beta}$ have Gaussian beta ensembles \cite{Dumitriu-Edelman-2002}, that is, 
\[
	(\lambda_1,\lambda_2, \dots, \lambda_n) \propto  |\Delta(\lambda)|^\beta \exp \left(-\frac{n\beta}{4}\sum_{i = 1}^n \lambda_j^2 \right).
\]
The weights $\{w_i\} = \{q_i^2\}$ are independent of $\{\lambda_i\}$ and have Dirichlet distribution with parameters $(\beta/2, \dots, \beta/2)$, that is,  
\[
	(w_1, \dots, w_{n - 1}) \propto  \prod_{i = 1}^n w_i^{\frac\beta 2 - 1} \mathbf 1_{\{w_1+\cdots + w_{n - 1} < 1, w_i > 0\}}, \quad w_n = 1 - (w_1 + \cdots + w_{n - 1}).
\]
Recall that the distribution of $\{w_i\}$ is the same as that of the vector 
\[
	\left(\frac{\chi_{\beta,1}^2}{\sum_{i = 1}^n \chi_{\beta,i}^2} , \dots, \frac{\chi_{\beta,n}^2}{\sum_{i = 1}^n \chi_{\beta,i}^2} \right),
\]
with $\{\chi_{\beta,i}^2\}_{i = 1}^n$ being i.i.d.~sequence of random variables having chi-squared distributions with $\beta$ degrees of freedom.

\begin{lemma}\label{lem:chi-distribution}
\begin{itemize}
	\item[\rm(i)] As $k \to \infty$,
		\[
			\frac{\chi_k}{\sqrt{k}} \to 1 \text{ in probability and in $L_q$ for all $q<\infty$.} 
		\]
	A sequence $\{\chi_{k_n}/\sqrt{k_n}\}$ converges almost sure to $1$, if $\sum_{n = 1}^\infty e^{-\varepsilon k_n} < \infty$ for any $\varepsilon > 0$.
	\item[\rm(ii)] As $k \to \infty$,
		\[
			\sqrt{k}\left(\frac{\chi_k}{\sqrt{k}} - 1 \right) = (\chi_k - \sqrt{k})\dto \Normal(0, \frac 12).
		\]
\end{itemize}
\end{lemma}
\begin{proof}
	The convergence in probability and a central limit theorem are standard results. Let us prove the almost sure convergence part. For almost sure convergence, usually there is some relation between random variables in the sequence. If there is no relation, the following criterion, a direct consequence of the Borel-Cantelli lemma, becomes useful. The sequence $X_n$ converges almost surely to $x$ as $n \to \infty$, if for any $\varepsilon > 0$,
\[
	\sum_{n = 1}^{\infty} \Prob(|X_n - x| > \varepsilon) < \infty.
\]

For the proof here and later for Beta distributions, we use the following bounds (see Lemma~4.1 in \cite{Nagel-082013})
\[
	\Prob\left(\left| \frac{\chi_k^2}{k} - 1 \right| > \varepsilon  \right) \le 2 e^{-\frac{k \varepsilon^2}{8}},
\]
\[
	\Prob\left(\left| {\rm Beta}(x, y) -  \E[{\rm Beta}(x, y)] \right| > \varepsilon  \right) \le 4 e^{-\frac{\varepsilon^2}{128}\frac{x^3 + y^3}{xy}}.
\]
Here ${\rm Beta} (x,y)$ denotes the beta distribution with parameters $x$ and $y$. Then the proof follows immediately from the criterion above.
\end{proof}

Since the Jacobi parameters $\{a_i, b_i\}$ of $H_{n,\beta}$ are independent, it follows that joint convergence in distribution holds. Thus we can write  
\[
	H_{n,\beta} \approx 
	\begin{pmatrix}
		    0 	&1	\\
		    1	&0	&1	\\
				&\ddots &\ddots &\ddots
	\end{pmatrix}
	+ \frac{1}{\sqrt{\beta n}}
	\begin{pmatrix}
		    \Normal(0,2)		&\Normal(0,\frac12)	\\
		    \Normal(0,\frac12)		&\Normal(0,2)		&\Normal(0,\frac12)\\
		    &\ddots &\ddots &\ddots
	\end{pmatrix}.
\]
The nonrandom Jacobi matrix in the above expression is called the free Jacobi matrix, denoted by $J_{free}$, whose spectral measure is the semicircle distribution 
\[
	sc(x) = \frac{1}{2\pi} \sqrt{4 - x^2}, (-2 \le x \le 2).
\] 
There are several ways to derive that fact. For example, in the theory of orthogonal polynomials on the real line, it is derived from the relation of Chebyshev polynomials of the second kind \cite[Section~1.10]{Simon-book}. One can also calculate moments and show that they match moments of the semicircle distribution. However, we introduce here a method to find the spectral measure by calculating its Stieltjes transform, which is applicable to the next two cases.

Let $J$ be a Jacobi matrix with bounded coefficients,
\[
	J = \begin{pmatrix}
		    a_1 	&b_1		\\
		    b_1	&a_2		&b_2	\\
				&\ddots &\ddots &\ddots
	\end{pmatrix}.
\]
Then the spectral measure $\mu$ is unique and has compact support. Let $S_\mu$ be the Stieltjes transform of $\mu$,
\[
	S_\mu(z)	= \int \frac{d\mu(x)}{x - z}.
\]
In the theory of Jacobi matrices, the Stieljes transform $S_\mu$ is called an $m$-function,
\[
	S_\mu(z) = m(z) = \langle (J - z)^{-1}e_1, e_1 \rangle = (J-z)^{-1}(1,1).
\]
For bounded Jacobi matrix $J$, the $m$-function is an analytic function on $\C \setminus I$, for some bounded interval $I \subset \R$, $\Image m(z) > 0,$ if $\Image z > 0$, and $m(\bar z) = \overline{m(z)}$.

Let $J_1$ be a Jacobi matrix obtained by removing the top row and left-most column of $J$ and let $m_1(z)$ be the $m$-function of $J_1$. Then the following relation holds (see \cite[Theorem~3.2.4]{Simon-book})
\begin{equation}\label{m-function-relation}
	-\frac{1}{m(z)} = z - a_1 + b_1^2 m_1(z).
\end{equation}
When $a_i \equiv a$ and $b_i \equiv b$, then $m_1(z) = m(z)$, and hence $m(z)$ satisfies the following equation
\[
	-\frac{1}{m(z)} = z - a + b^2 m (z),
\] 
which is equivalent to a quadratic form
\[
	b^2 m(z) ^2 + (z - a) m(z) + 1 = 0. 
\]
In particular, by solving the above equation with $a = 0$ and $b = 1$, we obtain the explicit formula for the $m$-function of the free Jacobi matrix,
\[
	m_{free}(z) = - \frac{1}{2} \left( z - \sqrt{z^2 - 4}\right),
\]
where the branch of the square-root is chosen to ensure that $\Image m(z) > 0$, if $\Image z > 0$. 
Then the spectral measure $\mu$ of $J_\infty$ can be easily calculated by an inverse formula (\cite[Theorem~2.4.3]{Anderson-book})
\[
	\mu(I) = \lim_{\varepsilon \searrow 0} \frac{1}{\pi} \int_I \Image m(x + i \varepsilon) dx,
\] 
provided that $I$ is an open interval with neither endpoint on an atom of $\mu$. Moreover, when the limit $\lim_{\varepsilon \searrow 0}  \Image m(x + i \varepsilon)$ exists and is finite for all $x \in \R$, the measure $\mu$ has density given by   
\[
	\mu(x) = \lim_{\varepsilon \searrow 0} \frac{1}{\pi}  \Image m(x + i \varepsilon).
\]
The spectral measure of the free Jacobi matrix is now easily derived from the above formula.

Back to the case  $a_i \equiv a$ and $b_i \equiv b$, and denote by $m_{a,b}(z)$ its $m$-function. Then  
\begin{equation}\label{m-ab}
	m_{a,b} (z) = (J_{a,b} - z)^{-1} (1,1) = (b J_{free} + a - z )^{-1} (1,1) =\frac{1}{b} m_{free} (\frac{z - a}{b}).
\end{equation}
Here 
\[
	J_{a,b} = \begin{pmatrix}
		    a 	&b		\\
		    b	&a		&b	\\
				&\ddots &\ddots &\ddots
	\end{pmatrix}, \quad J_{free} = J_{0,1}.
\]

Here is the main result in this section.
\begin{theorem}
	The spectral measure $\mu_n$ of $H_{n,\beta}$ converges weakly, almost surely, to the semicircle distribution as $n \to \infty$. Moreover, for any polynomial $p$, as $n \to \infty$, 
		\[
			\frac{\sqrt{n\beta}}{\sqrt2} (\langle \mu_n, p\rangle - \langle sc, p\rangle ) \dto \Normal(0, \sigma_p^2), 
		\]
where $\sigma_p^2$ is a constant. 
\end{theorem}


\begin{remark}
	Here we implicitly assume that the parameter $\beta$ is fixed. However, the result still holds if $\beta$ varies and $n\beta \to \infty$. In this case, by our method, the spectral measures $\mu_n$ converge weakly, in probability, to the semicircle distribution because we do not require the rate of convergence of $n\beta$ to infinity. 
\end{remark}

\section{Wishart beta ensembles or $\beta$-Laguerre ensembles}
Let $G$ be an $m \times n$ matrix whose entries are i.i.d.~standard Gaussian $\Normal(0,1)$ random variables. Then for $m \ge n$,  $(m^{-1} G^t G)$ is a Wishart real matrix whose eigenvalues are distributed as 
\[
	(\lambda_1, \dots, \lambda_n) \propto |\Delta(\lambda)| \prod_{i = 1}^n \lambda_i^{\frac12(m - n + 1) - 1} \exp\left( - \frac {m}2\sum_{i = 1}^n \lambda_i \right), \quad \lambda_i > 0.
\]
Note that Wishart complex matrices are defined similarly by considering i.i.d.~complex Gaussian matrices $G$. The eigenvalues also a nice density formula similar as above. Then Wishart beta ensembles are defined  to be  ensembles of $n$ points on $(0, \infty)$ with the following joint probability density function  
\[
	(\lambda_1, \dots, \lambda_n) \propto |\Delta(\lambda)|^\beta \prod_{i = 1}^n \lambda_i^{\frac{\beta}{2}(m - n + 1) - 1} \exp\left( - \frac {m \beta}2\sum_{i = 1}^n \lambda_i \right).
\]
Here we only require that $n-1 < m \in \R$.

Dumitriu and Edelman \cite{Dumitriu-Edelman-2002} gave the following matrix model for Wishart beta ensembles. For $n \in \N$ and $m > n-1$, let $B_{\beta}$ be a bidiagonal matrix whose elements are independent and are distributed as
\[
	B_\beta = \frac{1}{\sqrt{m\beta}}\begin{pmatrix}
		\chi_{\beta m}	\\
		\chi_{\beta(n - 1)}	&\chi_{\beta m - \beta}	\\
						&\ddots			&\ddots	\\
						&				&\chi_\beta 	&\chi_{\beta m - \beta(n - 1)}
	\end{pmatrix}.
\]
Let $L_{m,n,\beta} = B_\beta B_\beta^t$. Then the eigenvalues of $L_{m,n,\beta}$ are distributed as Wishart beta ensembles. Moreover, the weights $\{w_i\}$ in the spectral measures of $L_{m, n, \beta}$ have Dirichlet distribution with parameters $(\beta/2, \dots, \beta/2)$ and are independent of $\{\lambda_i\}$. For empirical distributions,  the weak convergence to Marchenko-Pastur distributions and a central limit theorem for polynomial test functions were established, see \cite{Dumitriu-Edelman-2006} and the references therein.

Denote by $\{c_i\}_{i = 1}^n$ and $\{d_j\}_{j = 1}^{n - 1}$ the diagonal and the sub-diagonal of the matrix $\sqrt{m\beta}B_\beta$. Then 
\[
	L_{m,n,\beta} = \frac{1}{m\beta} \begin{pmatrix}
		c_1^2		&	c_1 d_1	\\
		c_1 d_1		&	c_2^2 + d_1^2		&	c_2 d_2	\\
					&	\ddots			&	\ddots		&	\ddots	\\
					&					&	c_{n - 1}d_{n - 1}	&	c_{n}^2 + d_{n - 1}^2 
	\end{pmatrix}.
\]
\begin{lemma}
For fixed $i$, as $n \to \infty$ with $n/m \to \gamma \in (0,1)$,
\begin{align*}
	\frac{c_i}{\sqrt{m\beta}}  &= \frac{\chi_{\beta(m - i + 1)}}{\sqrt{m\beta}} \approx 1 + \frac{\sqrt{\gamma}}{\sqrt{n\beta}} \eta_i, \quad \eta_i \sim \Normal(0,\frac12),\\
	\frac{d_i}{\sqrt{m\beta}}  &= \frac{\chi_{\beta(n - i)}}{\sqrt{m\beta}} \approx \sqrt{\gamma} + \frac{\sqrt{\gamma}}{\sqrt{n\beta}} \zeta_i, \quad \zeta_i\sim \Normal(0,\frac12),\\
	\frac{c_i^2}{{m\beta}}  &\approx 1 + \frac{\sqrt{\gamma}}{\sqrt{n\beta}} 2\eta_i,\quad
	\frac{d_i^2}{{m\beta}}   \approx {\gamma} + \frac{\sqrt{\gamma}}{\sqrt{n\beta}}2 \zeta_i, \\
	\frac{c_i d_i}{m\beta} &\approx \sqrt{\gamma} +  \frac{\sqrt{\gamma}}{\sqrt{n\beta}}(\sqrt{\gamma} \eta_i + \zeta_i).
\end{align*}
\end{lemma}
Consequently, as $n \to \infty$ and $n/m \to \gamma \in (0,1)$, we can write 
\[
	L_{m,n,\beta} \approx \begin{pmatrix}
				1		&\sqrt{\gamma}	\\
				\sqrt{\gamma}	&1+\gamma	&\sqrt{\gamma}\\
							&\ddots		&\ddots		&\ddots
			\end{pmatrix}
		+ 
		\frac{\sqrt{\gamma}}{\sqrt{n \beta }}
		\begin{pmatrix}
				2 \eta_1	&	\sqrt{\gamma} \eta_1 + \zeta_1		\\
				\sqrt{\gamma} \eta_1 + \zeta_1	& 2(\eta_2 + \zeta_1)	&	\sqrt{\gamma} \eta_2 + \zeta_2	\\
								&\ddots		&\ddots		&\ddots
		\end{pmatrix}.
\]
Here $\{\eta_i\}$ and $\{\zeta_i\}$ are two i.i.d.~sequences of  $\Normal(0, \frac 12)$ random variables.

For $\gamma \in (0,1)$, let 
	\[
		MP_\gamma =\begin{pmatrix}
				1		&\sqrt{\gamma}	\\
				\sqrt{\gamma}	&1+\gamma	&\sqrt{\gamma}\\
							&\ddots		&\ddots		&\ddots
			\end{pmatrix} ,
	\]
and denote its $m$-function by $m_\gamma(z)$. When the first row and the left-most column are removed, we obtain the Jacobi matrix $J_{1+\gamma, \sqrt{\gamma}}$. Then it follows from the relations~\eqref{m-function-relation} and \eqref{m-ab} that 
\begin{align*}
	m_\gamma(z) &= -\frac{1}{z - 1 + \sqrt{\gamma}m_{free} (\frac{z - (1 + \gamma)}{ \sqrt{\gamma}})} \\
	 &= -\frac{1}{z - 1 - \frac{\sqrt{\gamma}}{2} \left(\frac{z - (1 +\gamma)}{\sqrt{\gamma}} - \sqrt{(\frac{z - (1+ \gamma)}{\sqrt\gamma})^2 - 4} \right)} \\
	 &= \frac{1 - \gamma - z + \sqrt{(z - \lambda_+)(z - \lambda_-)}}{2\gamma z},
\end{align*}
where $\lambda_{\pm} = (1 \pm \sqrt{\gamma})^2$. By taking the limit $\lim_{\varepsilon \searrow 0} \pi^{-1} \Image m_\gamma(x + i \varepsilon)$, we get the density of the spectral measure of $MP_\gamma$, which coincides with the density of the Marchenko-Pastur distribution with parameter $\gamma \in (0,1)$, 
	\[
		mp_\gamma(x) = 
				\frac{1}{2 \pi \gamma x} \sqrt{ (\lambda_+ - x)(x - \lambda_-)},  (\lambda_-<x< \lambda_+).
	\]

\begin{theorem}
	The spectral measure $\mu_n$ of $L_{m,n,\beta}$ conveges weakly, almost surely, to the Marchenko-Pastur distribution with parameter $\gamma$ as $n \to \infty$, $n/m \to \gamma \in (0,1)$. Moreover for any polynomial $p$,
	\[
		\frac{\sqrt{n\beta}}{\sqrt2} (\langle \mu_n, p\rangle - \langle mp_\gamma, p\rangle ) \dto \Normal(0, \sigma_p^2),
	\]
where $\sigma_p^2$ is a constant. 
\end{theorem}

\section{MANOVA beta ensembles or $\beta$-Jacobi ensembles}
Let $W_1 = G_1^t G_1$ and $W_2 = G_2^t G_2$ be two independent real Wishart matrices with $G_1$ (resp.~$G_2$) being an $m_1\times n$ (resp.~$m_2\times n)$ matrix whose entries are i.i.d.~standard Gaussian $\Normal(0,1)$ random variables. Assume that $m_1, m_2 \ge n$. Then the eigenvalues $(\lambda_1, \dots, \lambda_n)$ of the matrix $W_1(W_1 + W_2)^{-1}$ (or of the Hermitian matrix $(W_1 + W_2)^{-1/2} W_1 (W_1 + W_2)^{-1/2}$) are distributed according to the following joint probability density function with $\beta = 1$ (\cite[Theorem~3.3.4]{Muirhead-book})
\begin{align}
	(\lambda_1, \dots, \lambda_n) &\propto |\Delta(\lambda)|^\beta \prod_{i = 1}^n \lambda_i^{\frac{\beta}{2}(m_1 - n + 1) - 1} (1 - \lambda_i)^{\frac{\beta}{2}(m_2 - n + 1) - 1}, \nonumber\\
	&= |\Delta(\lambda)|^\beta \prod_{i = 1}^n \lambda_i^{a} (1 - \lambda_i)^{b}, \quad \lambda_i \in [0,1]. \label{Jacobi-ensembles}
\end{align}
For general $\beta>0$, the above joint density is referred to as MANOVA beta ensembles or $\beta$-Jacobi ensembles with parameters $(a,b)$, where  $a, b>-1$.

A matrix model for $\beta$-Jacobi ensembles is given as follows. Let $a, b > -1$ be fixed. Let $p_1, \dots, p_{2n - 1}$ be independent random variables distributed as 
\[
	p_k \sim \begin{cases}
		{\rm Beta}\left(\frac{2n - k}{4}\beta, \frac{2n - k -2}{4}\beta + a + b + 2 \right), &\text{$k$ even},\\
		{\rm Beta}\left(\frac{2n - k - 1}{4}\beta + a + 1, \frac{2n - k -1}{4}\beta +  b + 1 \right), &\text{$k$ odd}.
	\end{cases}
\]
Here recall that ${\rm Beta}(x, y)$ denotes the beta distribution with parameters $x, y$. Define 
\begin{align*}
	a_k &= p_{2k - 2} (1 - p_{2k - 3}) + p_{2k - 1}(1 - p_{2k - 2}), \\
	b_k &=\sqrt{p_{2k - 1}(1 - p_{2k - 2}) p_{2k} (1 - p_{2k - 1})}, 
\end{align*}
where $p_{-1} = p_0 = 0$, and form a random Jacobi matrix $J_{n,\beta}(a,b)$ as  
\[
	J_{n,\beta}(a,b) = \begin{pmatrix}
		a_1		&	b_1	\\
		b_1		&	a_2		&b_2\\
		&\ddots	&\ddots		&\ddots\\
		&&		b_{n - 1}		&a_n
	\end{pmatrix}.
\]
Then the eigenvalues $(\lambda_1, \dots, \lambda_n)$ of $J_{n,\beta}(a, b)$ are distributed as the $\beta$-Jacobi ensembles~\eqref{Jacobi-ensembles} (cf.~\cite{Killip-Nenciu-2004}).
The weights $\{w_i\} = \{q_i^2\}$ are independent of $\{\lambda_i\}$ and have Dirichlet distribution with parameters $(\beta/2, \dots, \beta/2)$. We are interested in studying $\beta$-Jacobi ensembles in the regime that $a(n)/(\frac{n \beta}{2}) \to \kappa_a  \in [0, \infty)$ and $b(n)/(\frac{n \beta}{2}) \to \kappa_b \in [0, \infty)$. The limiting behavior of empirical distributions in that regime was studied in \cite{Dumitriu-Paquette-2012}.

We need the following properties of beta distributions.
\begin{lemma}\label{lem:Beta-distribution}
Let $\{x_k\}$ and $\{y_k\}$ be two sequences of positive real numbers such that as $k \to \infty$
\[
	\frac{x_k}{k} \to x > 0, \frac{y_k}{k}  \to y > 0. 
\]
Then the following asymptotic behaviors of the beta distribution ${\rm Beta}(x_k, y_k)$ hold.
\begin{itemize}
\item[\rm(i)] As $k \to \infty$,
\[
	{\rm Beta}(x_k,y_k) \to \frac{x}{x + y} \text{ in probability and in $L_q$ for all $q<\infty$.}
\]
The almost sure convergence also holds. 
\item[\rm(ii)] As $k \to \infty$,
\[
\sqrt{k}\left({\rm Beta}(x_k, y_k) - \frac{x_k}{x_k +y_k}\right) \dto \Normal(0,x y(x +y)^{-3}).
\]
Moreover, if $(x_k y - y_k x)/\sqrt{k} \to 0$ as $k \to \infty$, then 
\[
	\sqrt{k}\left({\rm Beta}(x_k, y_k) - \frac{x}{x + y}\right) \dto \Normal(0,xy(x + y)^{-3}).
\]
\end{itemize}
\end{lemma}
\begin{proof}
	Let $X_k$ and $Y_k$ be two independent random variables distributed as $\chi^2_{2x_k}$ and $\chi^2_{2y_k}$, respectively. Then 
\[
	{\rm Beta}(x_k, y_k) \overset{d}{=} \frac{X_k}{X_k + Y_k}.
\]
For chi-squared distribution, we have 
\[
	\frac{\chi^2_k}{k} \to 1 \text{ in probability as }k \to \infty.
\]
Therefore 
\[
	\frac{X_k}{X_k + Y_k} = \frac{\frac{X_k}{2x_k}\frac{2x_k}{2k}}{\frac{X_k}{2x_k} \frac{2x_k}{2k}  + \frac{Y_k}{2y_k} \frac{2y_k}{2k}} \to \frac{x}{x + y} \text{ in probability as } k \to \infty.
\]
The convergence in $L_q$ is clear because beta distributions are bounded by $1$. For the almost sure convergence, see the proof of Lemma~\ref{lem:chi-distribution}.

Next we consider the central limit theorem for beta distributions. It also follows from the following central limit theorem for chi-squared distribution 
\[
	\sqrt{\frac k2} \left( \frac{\chi_k^2}{k} - 1  \right) \dto \Normal(0, 1) \text{ as }k \to \infty.
\]
Indeed, if we write 
\begin{align*}
	&\sqrt{k}\left(\frac{X_k}{X_k + Y_k} - \frac{x_k}{x_k + y_k}\right) \\
	&{=} 
	\frac{\left(\frac{y_k}{k} \frac{\sqrt{x_k}}{\sqrt{k}} \right) \sqrt{x_k} \left( \frac{X_k}{2x_k} - 1 \right) - \left(\frac{x_k}{k} \frac{\sqrt{y_k}}{\sqrt{k}} \right) \sqrt{y_k} \left( \frac{Y_k}{2y_k} - 1 \right)} {\left(\frac{x_k} k + \frac{y_k}k \right)\left(\frac{X_k}{2k} + \frac{Y_k}{2k} \right)},
\end{align*}
then as $k \to \infty$, the numerator converges in distribution to $\Normal(0,xy(x + y))$ because $X_k$ and $Y_k$ are independent while the denominator converges in probability to $(x + y)^2$. Thus we obtain (ii). 
\end{proof}

%
%

\begin{lemma}
As $n \to \infty$ with $a(n) = \frac{n\beta}{2} \kappa_a + o((n\beta)^{1/2})$ and $b(n) = \frac{n\beta}{2} \kappa_b + o((n\beta)^{1/2})$,
\begin{align*}
	p_{2k} & \approx {\rm Beta} (\frac{n \beta}{2}, \frac{n\beta}{2}(1 + \kappa_a + \kappa_b)) \approx \frac{1}{2 + \kappa_a + \kappa_b} + \frac{1}{\sqrt{n \beta}} \Normal(0, \sigma_{even}^2),\\
	p_{2k - 1} &\approx {\rm Beta} (\frac{n \beta}{2}(1 + \kappa_a), \frac{n\beta}{2}(1 + \kappa_b)) \approx \frac{1+\kappa_a}{2 + \kappa_a + \kappa_b} + \frac{1}{\sqrt{n \beta}} \Normal(0, \sigma_{odd}^2).
\end{align*}
As a consequence, 
\begin{align*}
	a_1^{(n)} &= p_1 \approx \frac{1+\kappa_a}{2 + \kappa_a + \kappa_b}  + \frac{1}{\sqrt{n\beta}} \Normal =: A_1 + \frac{1}{\sqrt{n\beta}} \Normal,\\
	b_1^{(n)} &= \sqrt{p_1 p_2 (1 - p_1)} \approx \frac{\sqrt{(1 + \kappa_a) (1 + \kappa_b)}}{(2 + \kappa_a + \kappa_b)^{3/2}}+ \frac{1}{\sqrt{n\beta}} \Normal =: B_1 + \frac{1}{\sqrt{n\beta}} \Normal ,\\
	a_k^{(n)} &= p_{2k - 2}(1 - p_{2k-3}) + p_{2k - 1}(1 - p_{2k - 2}) \\
	&\approx \frac{1 + \kappa_b + (1 + \kappa_a)(1 + \kappa_a + \kappa_b)}{(2 + \kappa_a + \kappa_b)^2} 
 + \frac{1}{\sqrt{n\beta}} \Normal
 	=:A + \frac{1}{\sqrt{n\beta}} \Normal, \quad k\ge 2,\\
	b_k^{(n)} &= \sqrt{p_{2k - 1}(1 - p_{2k - 2}) p_{2k}(1 - p_{2k - 1})} \\
	&\approx \frac{\sqrt{(1 + \kappa_a) (1 + \kappa_b)(1 + \kappa_a + \kappa_b)}}{(2 + \kappa_a + \kappa_b)^{2}} + \frac{1}{\sqrt{n\beta}}\Normal 
	=: B + \frac{1}{\sqrt{n\beta}} \Normal,  \quad k\ge 2.
\end{align*}
Here $\Normal$ denotes a normal distribution with mean $0$ and positive variance.
The joint asymptotic also holds.
\end{lemma}

\begin{proof}
For fixed $k$, as $n \to \infty$, the asymptotic for $p_k$ follows from Lemma~\ref{lem:Beta-distribution}. 
The asymptotic for $a_k^{(n)}$ follows from Lemma~\ref{lem:polynomial} because it is a polynomial of $\{p_{2k-3},p_{2k-2}, p_{2k-1}\}$.

The asymptotic for $b_k^{(n)}$ is a consequence of the following fact. If $X_n \to c \ne 0$ in probability and $\sqrt{n}(X_n - c) \dto \Normal(0, \sigma^2)$ as $n \to \infty$, then
		\[
			\sqrt{n}(\sqrt{X_n} - \sqrt{c}) = \frac{\sqrt{n}(X_n - c)}{\sqrt{X_n} + \sqrt{c}} \dto \frac{\Normal (0, \sigma^2)}{2\sqrt{c}} = \Normal (0, \frac{\sigma^2}{4c}).
		\]
The joint asymptotic is clear. 
\end{proof}

Therefore, in the regime that $a(n) = \frac{n\beta}{2} \kappa_a + o((n\beta)^{1/2})$ and $b(n) = \frac{n\beta}{2} \kappa_b + o((n\beta)^{1/2})$, the limit Jacobi matrix is of the form
\[
		J_\infty = \begin{pmatrix}
		A_1		&B_1	\\
		B_1	&A	&B	\\
		&B	&A	&B	\\
		&&\ddots		&\ddots		&\ddots
		\end{pmatrix},
\]
where $A_1, B_1, A$ and $B$ are defined in the above lemma.
We use relations~\eqref{m-function-relation} and \eqref{m-ab} again to derive an explicit formula for the $m$-function of $J_\infty$, 
\[
	m(z) = \frac{\kappa_a}{2z} + \frac{\kappa_b}{2(z - 1)}  -\frac{(2+\kappa_a + \kappa_b) \sqrt{(z- u_-)(z - u_+)}}{2 z ( z -1)},
\]
where
\begin{equation}
	u_{\pm} = A \pm 2B = \left(\frac{\sqrt{(1 + \kappa_a)(1 + \kappa_a + \kappa_b)} \pm \sqrt{1 + \kappa_b}}{2 + \kappa_a + \kappa_b} \right)^2.
\end{equation}
Note that $0 \le u_- < u_+ \le 1$. The density of the spectral measure of $J_\infty$ can be easily calculated by the inverse formula,
\[
	km_{u_, u_+}(x) = \frac{2 + \kappa_a + \kappa_b}{2 \pi} \frac{\sqrt{(x - u_-)(u_+ - x)}}{ {x(1-x)}}, (u_- < x < u_+).
\]
It coincides with the density of the Kesten-McKey distribution with parameters $(u_-, u_+)$ (see \cite[Section~7.6]{Gamboa-Rouault-2011}) because  
\[
 \frac{1}{2}(1 - \sqrt{u_- u_+} - \sqrt{(1 - u_-)(1- u_+)}) = \frac{1}{2 + \kappa_a + \kappa_b}.
\]
When $\kappa_a = \kappa_b = 0$, then $u_-=0$ and $u_+=1$, we get the arcsine distribution. 

Therefore, we obtain the limiting behavior of the spectral measures of $\beta$-Jacobi ensembles.  \begin{theorem}
As $n \to \infty$ with $a(n) = \frac{n\beta}{2} \kappa_a + o((n\beta)^{1/2})$ and $b(n) = \frac{n\beta}{2} \kappa_b + o((n\beta)^{1/2})$, the spectral measure $\mu_n$ of $J_{n,\beta}(a(n), b(n))$ converges weakly, almost surely, to the Kesten-McKey distribution with parameters $(u_-, u_+)$. For any polynomial $p$, 
\[
	\frac{\sqrt{n\beta}}{\sqrt2} ( \langle \mu_n, p\rangle - \langle km_{u_-, u_+}, p\rangle)  \to \Normal(0, \sigma_p^2),
\]
where $\sigma_p^2$ is a constant. 
\end{theorem}
\begin{remark}
	In different regimes, the weak convergence of both empirical distributions and spectral measures was considered in \cite{Nagel-082013}, in which the limit distribution is the Marchenko-Pastur distribution or the semicircle distribution.   
\end{remark}

\section{Extend the central limit theorem to a large class of test functions}
For all three beta ensembles in this paper, the spectral measure $\mu_n$ can be written as 
\[
	\mu_n = \sum_{i = 1}^n w_i \delta_{\lambda_i},
\]
where the weights $\{w_i\} $ are independent of the eigenvalues $\{\lambda_i\}$ and have Dirichlet distribution with parameters $(\beta/2, \dots, \beta/2)$. Recall that as a consequence of Theorem~\ref{thm:general}, for any polynomial $p$, as $n \to \infty$,
\[
	\langle \mu_n, p\rangle \to \langle \mu_\infty, p\rangle \text{ almost surely and in $L^q$ for $q\ge 1$},
\] 
\[
	\frac{\sqrt{n\beta}}{{\sqrt 2}} (\langle \mu_n, p\rangle  - \E[\langle \mu_n, p\rangle]) \dto \Normal (0, \sigma^2_p),
\]
where 
\[
	\sigma^2_p = \lim_{n \to \infty} \frac{n \beta}{2} \Var[\langle \mu_n, p\rangle ],
\]
and $\mu_\infty$ denotes the limit distribution.

One can easily show that 
\begin{align*}
	\E[w_i] &= \frac 1n,
	\E[w_i^2] = \frac{\beta+2}{n(n\beta + 2)},
	\E[w_i w_j] = \frac{\beta}{n(n \beta +2)}, (1\le i \neq j \le n).
\end{align*}
Therefore for any test function $f$, the following relations hold
\begin{align}
	\E[\langle \mu_n, f\rangle] &= \E[\langle L_n, f\rangle], \label{mean-measure}\\
	\Var[\langle \mu_n, f\rangle]&=\frac{\beta n}{\beta n + 2} \Var[\langle L_n, f \rangle] + \frac{2}{n\beta + 2} \left (\E[\langle \mu_n, f^2 \rangle] - \E[\langle \mu_n, f \rangle]^2 \right).\label{variance-relation}
\end{align} 

The mean of a random measure $\mu$, denoted by $\bar \mu$, is defined to be a probability measure satisfying 
\[
	\langle \bar \mu, f \rangle = \E[\langle \mu, f \rangle], \]
for all bounded continuous function $f$. Moreover, the above relation still holds for any continuous function $f$ with $\E[\langle \mu, |f|\rangle] < \infty$. It follows from \eqref{mean-measure} that $\bar L_n = \bar \mu_n$, and hence, $\bar L_n = \bar\mu_n$ converges weakly to $\mu_\infty$ as $n \to \infty$. Thus for any continuous function of polynomial growth, 
\[
	\langle \bar \mu_n, f\rangle \to \langle \mu_\infty, f\rangle \text{ as } n \to \infty.
\]

Denote by $C(\R)$ the set of continuous functions on $\R$ and let 
\[
	\mathcal D = \{f \in C(\R): n\Var[\langle L_n, f \rangle] \to 0, \langle \bar \mu_n, f \rangle \to \langle \mu_\infty, f\rangle, \langle \bar\mu_n, f^2\rangle \to \langle \mu_\infty, f^2\rangle\}.
\]
Then $\mathcal D$ is a linear space containing all polynomials. It follows from the variance relation \eqref{variance-relation} that for $f \in \mathcal D$,
\[
	\lim_{n \to \infty} \frac{n \beta}{2} \Var[\langle \mu_n, f\rangle] = \langle \mu_\infty, f^2\rangle - \langle \mu_\infty, f\rangle^2 =: \sigma^2(f).
\]

Next, we use the following result to extend the central limit theorem to any test function in $\mathcal D$.

\begin{lemma}[{\cite[Theorem 25.5]{Billingsley-PnM}}]\label{lem:triangle}

Let $\{Y_n\}_n$ and $\{X_{n,k}\}_{k,n}$ be real-valued random variables. Assume that
\begin{itemize}
	\item[\rm(i)] 
		$
			X_{n,k} \dto X_k \text{ as }n \to \infty;
		$
	\item[\rm(ii)]
		$
			X_k \dto X \text{ as } k \to \infty;
		$
	\item[\rm(iii)] for any $\varepsilon > 0$,
		$
			\lim_{k \to \infty} \limsup_{n \to \infty} \Prob(|X_{n,k} - Y_n| \ge \varepsilon) =0.
		$
\end{itemize}
Then $Y_n \dto X$ as $n \to \infty$.
\end{lemma}

\begin{theorem}
	For $f \in \mathcal D$, 
	\[
		\frac{\sqrt{n\beta}}{\sqrt{2}} \big(\langle \mu_n, f\rangle - \E[\langle \mu_n, f\rangle] \big) \dto \Normal(0, \sigma(f)^2),
	\]
	where $\sigma^2(f) = \langle \mu_\infty, f^2\rangle - \langle \mu_\infty, f\rangle^2 = \Var_{\mu_\infty}[f]$.
\end{theorem}
\begin{proof}
Let $f \in \mathcal D$. Since $\mu_\infty$ has compact support, we can find a sequence of polynomials $\{p_k\}$ converging to $f$ uniformly in the support of $\mu_\infty$. Thus
\begin{equation*}\label{variance-approximation}
	\sigma^2(p_k) \to \sigma^2(f) \text{ as }k \to \infty.
\end{equation*}
Let 
\begin{align*}
	Y_n &= \frac{\sqrt{n\beta}}{\sqrt{2}} (\langle \mu_n, f \rangle - \E[\langle \mu_n, f \rangle] ),\\
	X_{n,k} &= \frac{\sqrt{n\beta}}{\sqrt{2}} (\langle \mu_n, p_k \rangle - \E[\langle \mu_n, p_k \rangle] ).
\end{align*}
We only need to check three conditions in Lemma~\ref{lem:triangle}. Conditions (i) and (ii) are clear. For the 
condition (iii), note that $(f - p_k) \in \mathcal D$, and thus
\[
	\lim_{n \to \infty}\Var[X_{n,k} - Y_n] = \langle \mu_\infty, (f-p_k)^2\rangle - \langle \mu_\infty, (f-p_k)\rangle^2,
\]
which tends to zero as $k \to \infty$.
Therefore, for any $\varepsilon > 0$,
		\begin{align*}
			\lim_{k \to \infty} \limsup_{n \to \infty} \Prob(|X_{n,k} - Y_n| \ge \varepsilon) \le \lim_{k \to \infty} \limsup_{n \to \infty} \frac{1}{\varepsilon^2} \Var[X_{n,k} - Y_n] = 0.
		\end{align*}
The theorem is proved.
\end{proof}

\begin{lemma}\label{lem:variance-control}
For the Gaussian and Wishart beta ensembles, the class $\mathcal D$ contains at least all functions whose derivative is continuous of polynomial growth. For the MANOVA beta ensembles, the class $\mathcal D$ contains at least all differentiable functions with continuous derivative on $[0,1]$, provided that the parameters $a(n)$ and $b(n)$ are positive and $a(n) + b(n) \to \infty$ as $n \to \infty$.
\end{lemma}

The idea of proof is taken from \cite{Dumitriu-Paquette-2012}. The key tool is the following result.
\begin{lemma}[{\cite[Proposition~2.1]{Bobkov-Ledoux-2000}}]
Let $d\nu = e^{-V} dx$ be a probability measure supported on an open convex set $\Omega \subset \R^n$. Assume that $V$ is twice continuously differentiable and strictly convex on $\Omega$. Then for any locally Lipschitz function $F$ on $\Omega$,
\begin{align*}
	\Var_\nu[F] = \int F^2 d\nu - \left( \int F d\nu \right)^2 &\le \int \langle \Hess(V)^{-1} \nabla F, \nabla F \rangle d\nu \\
	&\le \int \frac{1}{\lambda_{\min} (\Hess(V))} | \nabla F|^2  d\nu,
\end{align*}
where $\Hess(V)$ denotes the Hessian of $V$ and $\lambda_{\min} (A)$ denote the smallest eigenvalue of $A$. 
\end{lemma}

\begin{proof}[Proof of Lemma~{\rm\ref{lem:variance-control}}]
Let us consider the Gaussian beta ensembles case first. Here we consider the ordered eigenvalues $\lambda_1 < \cdots < \lambda_n$. Let $\Omega = \{(\lambda_1, \dots, \lambda_n) : \lambda_1 < \lambda_2 < \cdots < \lambda_n  \} \subset \R^n$. Then the joint probability density function can be written in the form $e^{-V}$ on $\Omega$ with 
\[
	V = const +  \frac{n \beta}{4} \sum_{i = 1}^n \lambda_i^2  -  \frac{\beta}{2}\sum_{i \ne j}{\log| \lambda_j - \lambda_i|}.
\] 
It follows that 
\[
	\frac{\partial V}{\partial \lambda_i} = \frac{n\beta}{2} \lambda_i - \beta \sum_{j \ne i} \frac{1}{\lambda_i - \lambda_j},
\]
and hence
\begin{align*}
	\frac{\partial^2 V}{\partial\lambda_i^2} &= \frac{n \beta}{2} + \beta \sum_{j \ne i} \frac{1}{(\lambda_i - \lambda_j)^2},\\
	\frac{\partial^2 V}{\partial\lambda_i \partial \lambda_j} &= - \beta  \frac{1}{(\lambda_i - \lambda_j)^2}.
\end{align*}
By  the Gershgorin circle theorem, the smallest eigenvalue of $\Hess(V)$ is at least $n\beta /2$,
\[
	\lambda_{\min} (\Hess(V)) \ge \frac{n\beta}{2}.
\]
Therefore for any locally Lipschitz function $F$, 
\begin{equation}\label{general-F}
	\Var_\nu[F] \le \frac{2}{n\beta} \int |\nabla F|^2 d\nu.
\end{equation}
Now let $f$ be a continuous function on $\R$ with continuous derivative and let 
\[
	F(\lambda_1, \dots, \lambda_n) = \frac{1}{n} \sum_{i = 1}^n f(\lambda_i)  (= \langle L_n, f\rangle).
\]
Then it follows from \eqref{general-F} that 
\begin{equation}\label{GE}
	\Var[\langle L_n, f\rangle] \le \frac{2}{n^2 \beta} \int \langle L_n, (f')^2\rangle d\nu = \frac{2}{n^2 \beta}  \langle \bar L_n, (f')^2\rangle = \frac{2}{n^2 \beta}\langle \bar \mu_n, (f')^2\rangle. 
\end{equation}
Therefore when $f$ has continuous derivative of polynomial growth, as $n \to \infty$,
\[
	\langle \mu_n, (f')^2\rangle \to \langle sc, (f')^2\rangle.
\]
Consequently, $n\Var[\langle L_n, f\rangle] \to 0$, and thus the  class $\mathcal D$ contains all functions $f$ which have continuous derivative of polynomial growth.

For MANOVA beta ensembles, by similar argument we arrive at the following inequality
\[
	\Var[\langle L_n, f\rangle] \le \frac{1}{n(a(n) + b(n)) }  \langle \bar \mu_n, (f')^2\rangle,
\]
provided that the parameters $a(n)$ and $b(n)$ are positive. Thus  if $a(n) + b(n)$ tends to infinity, we also have 
\[
	n \Var[\langle L_n, f\rangle] \to 0 \text{ as } n \to \infty,
\]
for all functions $f$ which have continuous derivative on $[0,1]$.

The Wishart beta ensembles are little different because we do not have a uniform estimate for $\lambda_{\min}(\Hess(V))$. In this case we are working on $\Omega = \{(\lambda_1, \dots, \lambda_n) : 0< \lambda_1 < \lambda_2 < \cdots < \lambda_n  \} \in \R^n$ with  
\[
	V = const + \frac{m\beta}{2} \sum_{i = 1}^n \lambda_i - a \sum_{i = 1}^n\log \lambda_i -  \frac{\beta}{2}\sum_{i \ne j}{\log| \lambda_j - \lambda_i|}.
\]
Here $a = \frac{\beta}{2} (m - n + 1) - 1$.
Therefore
\begin{align*}
	\frac{\partial^2 V}{\partial\lambda_i^2} &= \frac{a}{\lambda_i^2} + \beta \sum_{j \ne i} \frac{1}{(\lambda_i - \lambda_j)^2},\\
	\frac{\partial^2 V}{\partial\lambda_i \partial \lambda_j} &= - \beta  \frac{1}{(\lambda_i - \lambda_j)^2}.
\end{align*}
By  the Gershgorin circle theorem again, we get the following bound 
\[	
	\lambda_{\min} (\Hess(V)) \ge \frac{a}{ \lambda_n^2}.
\]
Consequently, 
\[
	\Var_\nu[F] \le \frac{1}{a} \int \lambda_n^2 |\nabla F|^2 d\nu,
\]
and hence, 
\[
	\Var[\langle L_n, f \rangle] \le \frac{1}{a} \E \left[\frac{\lambda^2_{\max}}{n} \langle L_n, (f')^2\rangle \right] .
\]
Here $\lambda_{\max}$ denotes the largest eigenvalue.
Assume that $f'$ is dominated by some polynomial. Then $\langle L_n, (f')^2 \rangle $ is bounded in $L^2$ because of \eqref{variance-relation}. In addition, 
\[
	\E \left[\frac{\lambda_{\max}^4}{n^2}\right] \le \frac{1}{n} \E[\langle L_n, x^4\rangle] = \frac{1}{n} \langle \bar \mu_n, x^4 \rangle \to 0 \text{ as } n \to \infty.
\]
Finally, recall that $a = (\beta/2) (m - n + 1) -1$, which behaves like $\frac{\beta}{2}(\frac{1}{\gamma} - 1) n$ in the regime that $n/m \to \gamma \in (0,1)$. Thus 
\[
	n \Var[\langle L_n, f \rangle] \to 0 \text{ as } n \to \infty.
\]
The proof is complete.
\end{proof}

\bigskip
{\bf Acknowledgement. } The author would like to thank the referee for valuable comments.

\hfill\begin{tabular}{l}
Trinh Khanh Duy \\
Institute of Mathematics for Industry \\
Kyushu University\\
Fukuoka 819-0395, Japan \\
e-mail: trinh@imi.kyushu-u.ac.jp 
\end{tabular}

\end{document}